\documentclass[11pt]{article}
\usepackage[margin=1.15in]{geometry}
\usepackage{amsmath,amssymb,amsthm,mathtools,array}
\usepackage{microtype}
\usepackage[colorlinks=true,linkcolor=blue,citecolor=blue,urlcolor=blue]{hyperref}

\newtheorem{theorem}{Theorem}[section]
\newtheorem{proposition}[theorem]{Proposition}
\newtheorem{lemma}[theorem]{Lemma}
\newtheorem{corollary}[theorem]{Corollary}

\newtheorem{question}[theorem]{Open question}

\newcommand{\D}{\mathrm D}
\newcommand{\C}{\mathbb C}

\title{Monomial Moving Zeros for the Heat Semigroup:\\
An Explicit Finite-Order Differential Hierarchy}
\author{Gerardo Hern\'andez-del-Valle\\[2mm]
\small Centro de Estudios Monetarios Latinoamericanos (CEMLA), Mexico City,
Mexico\\[-1mm]
\small \texttt{ghernandez@cemla.org}}
\date{}

\begin{document}
\maketitle

\begin{abstract}
Let $P_t=\exp(t\D^2/2)$ be the heat semigroup.  For every integer
$k\geq2$ and parameter $\rho\in\C$, we construct an entire function
$\Gamma_{\rho,k}$ whose heat evolution has the prescribed moving zero
$x=\rho t^k$ and constant spatial flux there:
\[
 (P_t\Gamma_{\rho,k})(\rho t^k)=0,
 \qquad
 \partial_x(P_t\Gamma_{\rho,k})(\rho t^k)=1.
\]
The initial function is explicit,
\[
 \Gamma_{\rho,k}(x)
 =x\sum_{q\geq0}\frac{(-2^k\rho x^{2k-1})^q}{q!}
       \frac{(kq)!}{((2k-1)q+1)!},
\]
and is a generalized hypergeometric function whose entire-function order is
strictly below two.
Most importantly, it is annihilated by an explicit polynomial-coefficient
differential operator of order $2k-1$.  Thus every monomial moving boundary
admits a finite-order stationary realization after a natural boundary-flux
selection, even though the naive order-$k$ family
$\D^k+a x^{k-1}$ already fails to produce the exact boundary at $k=4$.
No minimality claim is made: the result gives a universal finite-order
hierarchy, while the least possible differential order remains open and
depends on the choice made inside the inverse kernel.
\end{abstract}



\section{Introduction}

Consider the heat equation
\begin{equation}\label{eq:heat}
 \partial_t v=\frac12\partial_x^2v,
 \qquad v(t,x)=(P_t\gamma)(x),
 \qquad P_t=\exp\!\left(\frac t2\D^2\right).
\end{equation}
The transported-zero problem asks for curves $f$ along which
\begin{equation}\label{eq:trace}
 (P_t\gamma)(f(t))=0.
\end{equation}
It arose for us from the construction of heat solutions associated with
Brownian boundary-crossing problems; see \cite{HDVGuerra2024}.  Moving
boundary traces and the associated Dirichlet-to-Neumann problem also have an
established analytic theory; see, for example, \cite{DeLilloFokas}.  The
classical heat-polynomial framework provides related background for the
action of the heat semigroup on polynomial data \cite{RosenbloomWidder}.  A direct
approach begins with a stationary differential equation for $\gamma$,
transports its differential operator through $P_t$, and derives conditions
on $f$.  Compatibility matrices, auxiliary Heat--Airy flows, and stationary
grading provide systematic versions of this procedure
\cite{HDVMatrix,HDVAiry,HDVGrading}.

The low-degree examples suggest the sparse family
\begin{equation}\label{eq:naive}
 \D^k+a x^{k-1}.
\end{equation}
After normalization, this family produces exact quadratic and cubic moving
zeros.  Stationary grading shows, however, that at $k=4$ the predicted
quartic is only the leading term of the boundary: higher graded corrections
are unavoidable for the relevant stationary modes \cite{HDVGrading}.
This failure motivates reversing the problem.  One prescribes $f$ and seeks
an initial function in the inverse kernel
\[
 \mathcal K(f)=\{\gamma:(P_t\gamma)(f(t))=0\}.
\]
The kernel is generally large, so the boundary alone does not choose a
unique stationary function or a unique differential equation
\cite{HDVInverse}.  The additional condition used here is the constant-flux
normalization
\begin{equation}\label{eq:fluxdata}
 (P_t\Gamma)(f(t))=0,
 \qquad \partial_x(P_t\Gamma)(f(t))=1.
\end{equation}

Our principal result is that for every monomial $f(t)=\rho t^k$, the
normalization \eqref{eq:fluxdata} selects an explicit entire function that
is annihilated by a finite-order differential operator.  Set
\[
 N=2k-1,
 \qquad \vartheta=x\D.
\]

\begin{theorem}[Monomial constant-flux hierarchy]\label{thm:main}
For every integer $k\geq2$ and $\rho\in\C$, define
\begin{equation}\label{eq:Gamma}
 \Gamma_{\rho,k}(x)
 =x\sum_{q=0}^{\infty}
 \frac{(-2^k\rho x^{N})^q}{q!}
 \frac{(kq)!}{(Nq+1)!}.
\end{equation}
Then:
\begin{enumerate}
\item $\Gamma_{\rho,k}$ is entire of order $N/k=2-1/k<2$ when
      $\rho\ne0$;
\item the heat series $V_{\rho,k}(t,x)=P_t\Gamma_{\rho,k}(x)$ converges
      locally uniformly on $\C^2$, solves
      $\partial_tV_{\rho,k}=\partial_x^2V_{\rho,k}/2$ with
      $V_{\rho,k}(0,x)=\Gamma_{\rho,k}(x)$, and satisfies
      \begin{equation}\label{eq:boundary-identities}
      (P_t\Gamma_{\rho,k})(\rho t^k)=0,
      \qquad
      \partial_x(P_t\Gamma_{\rho,k})(\rho t^k)=1;
      \end{equation}
\item $\Gamma_{\rho,k}$ satisfies the order-$N$ equation
      \begin{equation}\label{eq:mainode}
      \left\{
      N^{k-1}\D^{N}
      +k2^k\rho
      \prod_{r=1}^{k-1}
      \bigl[k\vartheta+Nr-k\bigr]
      \right\}\Gamma_{\rho,k}=0.
      \end{equation}
\end{enumerate}
\end{theorem}

The theorem does not assert that $N=2k-1$ is minimal.  In fact, the
quadratic and cubic boundaries also arise from different selections obeying
lower-order equations.  The point is instead that the obstruction to the
particular order-$k$ ansatz \eqref{eq:naive} is not an obstruction to
finite-order differential structure.  The inverse construction reveals the
operator rather than requiring it to be guessed in advance.

\section{Boundary jets and the selection principle}

Let $v$ solve \eqref{eq:heat} and put
\[
 q_n(t)=\partial_x^n v(t,f(t)).
\]
Differentiation along the curve gives the elementary recursion
\begin{equation}\label{eq:jetrec}
 q_{n+2}(t)=2q_n'(t)-2f'(t)q_{n+1}(t),
 \qquad n\geq0.
\end{equation}
Consequently, the two traces $q_0$ and $q_1$ formally determine every
spatial jet of $v$ on the curve.  For \eqref{eq:fluxdata},
\[
 q_0(t)=0,\qquad q_1(t)=1.
\]
At $t=0$, the jet values are the Taylor coefficients of the initial datum.
Thus constant flux is a genuine selection principle inside $\mathcal K(f)$.
For a general curve the resulting recursion need not have a compact closed
form.  The monomial case does, and the formula is precisely
\eqref{eq:Gamma}.

\begin{proposition}\label{prop:uniqueness}
Among formal power series at the origin, the two identities
\eqref{eq:boundary-identities} determine at most one initial function.
Hence \eqref{eq:Gamma}, once the identities are proved, is the unique formal
constant-flux selection for $f(t)=\rho t^k$.
\end{proposition}

\begin{proof}
Starting with $q_0=0$ and $q_1=1$, recursion \eqref{eq:jetrec} determines
$q_n$ successively.  Evaluating at $t=0$ determines every derivative
$\Gamma^{(n)}(0)$ and therefore the formal Taylor series.
\end{proof}

\section{Proof of the boundary identities}

We use the standard finite-difference fact
\begin{equation}\label{eq:fd}
 \sum_{q=0}^{r}(-1)^q\binom rq P(q)=0
 \quad\text{whenever}\quad \deg P<r.
\end{equation}

\begin{proposition}\label{prop:formaltrace}
The formal heat transform of \eqref{eq:Gamma} satisfies
\eqref{eq:boundary-identities}.
\end{proposition}

\begin{proof}
Write
\[
 \Gamma_{\rho,k}(x)=\sum_{q\geq0}c_qx^{Nq+1},
 \qquad
 c_q=\frac{(-2^k\rho)^q}{q!}\frac{(kq)!}{(Nq+1)!}.
\]
Substitute this series into
\[
 (P_t\Gamma)(\rho t^k)
 =\sum_{j\geq0}\frac{t^j}{2^jj!}
 \Gamma^{(2j)}(\rho t^k).
\]
For the term indexed by $(q,j)$, the power of $\rho$ is
\[
 q+Nq+1-2j=2kq+1-2j.
\]
Writing it as $2m+1$ forces $j=kq-m$; the corresponding power of $t$ is
then $k+Nm$.  Thus the coefficient of
$\rho^{2m+1}t^{k+Nm}$ is exactly
\begin{equation}\label{eq:zero-coefficient}
 \frac{2^m m!}{(2m+1)!}
 \sum_{q=0}^{2m+1}(-1)^q\binom{2m+1}{q}\binom{kq}{m}.
\end{equation}
Here $\binom{kq}{m}=0$ when $kq<m$, which incorporates the constraint
$j\geq0$.  Since $q\mapsto\binom{kq}{m}$ is a polynomial of degree
$m<2m+1$, \eqref{eq:fd} makes \eqref{eq:zero-coefficient} vanish.

For the spatial derivative, the term indexed by $(q,j)$ contains the power
$\rho^{2kq-2j}$.  Writing it as $\rho^{2m}$ again gives $j=kq-m$, and its
power of $t$ is $Nm$.  Consequently, the coefficient of
$\rho^{2m}t^{Nm}$ in $\partial_x(P_t\Gamma)(\rho t^k)$ is
\begin{equation}\label{eq:flux-coefficient}
 \frac{2^m m!}{(2m)!}
 \sum_{q=0}^{2m}(-1)^q\binom{2m}{q}\binom{kq}{m}.
\end{equation}
For $m\geq1$, the polynomial in the summand has degree $m<2m$, so
\eqref{eq:flux-coefficient} vanishes by \eqref{eq:fd}.  When $m=0$, the
sum has the single term $q=0$ and equals one.  Therefore the zero trace and
unit-flux identities both hold formally.
\end{proof}

\section{Entirety and convergence of heat transport}

\begin{proposition}\label{prop:order}
For $\rho\ne0$, the entire function \eqref{eq:Gamma} has order
$N/k$.  In particular its order is strictly smaller than two.
\end{proposition}

\begin{proof}
Stirling's formula applied to the nonzero coefficient $c_q$ gives
\[
 \log |c_q|^{-1}=kq\log q+O(q).
\]
The corresponding exponent is $Nq+1$.  The coefficient formula for the
order of an entire function therefore yields
\[
 \operatorname{ord}(\Gamma_{\rho,k})
 =\limsup_{q\to\infty}
 \frac{(Nq+1)\log(Nq+1)}{\log|c_q|^{-1}}
 =\frac Nk.
\]
The same estimate shows that the series has infinite radius of convergence.
\end{proof}

\begin{lemma}\label{lem:heat}
If $g$ is entire of order strictly smaller than two, then
\[
 \sum_{j\geq0}\frac{t^j}{2^jj!}g^{(2j)}(x)
\]
converges locally uniformly in $(t,x)\in\C^2$.
\end{lemma}

\begin{proof}
Choose $\sigma<2$ larger than the order of $g$.  On every fixed compact
$x$-set, Cauchy's estimate on circles of radius $R$ gives
\[
 |g^{(2j)}(x)|\leq (2j)!R^{-2j}C\exp(CR^\sigma).
\]
Taking $R=j^{1/\sigma}$ and using Stirling's formula, the logarithm of the
$j$th heat-series term is bounded above by
\[
 \left(1-\frac{2}{\sigma}\right)j\log j+O(j),
\]
uniformly for $(t,x)$ in compact sets.  Its leading coefficient is negative,
which proves normal convergence.
\end{proof}

Proposition \ref{prop:formaltrace}, Proposition \ref{prop:order}, and Lemma
\ref{lem:heat} prove parts (1) and (2) of Theorem \ref{thm:main}; normal
convergence promotes the formal identities to analytic identities.  Since
the sum is normally convergent, its derivatives may be taken term by term;
shifting the summation index then gives
\[
 \partial_t(P_t\Gamma_{\rho,k})
 =\frac12\partial_x^2(P_t\Gamma_{\rho,k}),
 \qquad P_0\Gamma_{\rho,k}=\Gamma_{\rho,k}.
\]

\begin{corollary}[Heat-kernel realization]\label{cor:convolution}
For real $t>0$ and $x\in\mathbb R$, the entire heat solution in
Theorem~\ref{thm:main} agrees with the Gaussian convolution
\begin{equation}\label{eq:convolution}
 V_{\rho,k}(t,x)
 =\frac{1}{\sqrt{2\pi t}}
 \int_{\mathbb R}
 \exp\!\left(-\frac{(x-y)^2}{2t}\right)
 \Gamma_{\rho,k}(y)\,dy.
\end{equation}
In particular, the convolution in \eqref{eq:convolution} vanishes at
$x=\rho t^k$ whenever $\rho\in\mathbb R$.
\end{corollary}

\begin{proof}
Because $\Gamma_{\rho,k}$ has order strictly below two, its growth on the
real axis is dominated by $C\exp(C|y|^\sigma)$ for some $\sigma<2$.
The Gaussian integral therefore converges absolutely.  Expanding
$\Gamma_{\rho,k}(x+\sqrt{t}\,z)$ into its Taylor series and integrating
term by term against the standard Gaussian is justified as follows.  For
$\sigma<2$ larger than the order, Cauchy's estimate with
$R=n^{1/\sigma}$ gives, locally uniformly in $x$,
\[
 \frac{|\Gamma_{\rho,k}^{(n)}(x)|}{n!}
 \leq \exp\!\left(-\frac{n}{\sigma}\log n+O(n)\right).
\]
On the other hand,
$\mathbb E|Z|^n=\exp((n/2)\log n+O(n))$.  Hence the series of integrated
absolute values converges because $1/2-1/\sigma<0$.  The odd moments vanish and
$\mathbb E[Z^{2j}]=(2j)!/(2^jj!)$, giving
\[
 \mathbb E\bigl[\Gamma_{\rho,k}(x+\sqrt{t}\,Z)\bigr]
 =\sum_{j\geq0}\frac{t^j}{2^jj!}
 \Gamma_{\rho,k}^{(2j)}(x)
 =V_{\rho,k}(t,x).
\]
The last assertion follows from \eqref{eq:boundary-identities}.
\end{proof}

\section{The finite-order annihilator}

\begin{proposition}\label{prop:ode}
The function \eqref{eq:Gamma} satisfies \eqref{eq:mainode}.
\end{proposition}

\begin{proof}
For the coefficients above,
\begin{equation}\label{eq:ratio}
 \frac{c_{q+1}}{c_q}
 =-k2^k\rho\,
 \frac{\prod_{s=1}^{k-1}(kq+s)}
 {\prod_{\ell=2}^{2k}(Nq+\ell)}.
\end{equation}
Apply the operator in \eqref{eq:mainode} to
$\sum_{q\geq0}c_qx^{Nq+1}$.  The coefficient of $x^{Nq+1}$ coming from
$\D^N$ is
\[
 N^{k-1}c_{q+1}\frac{(N(q+1)+1)!}{(Nq+1)!},
\]
whereas the Euler factors contribute
\[
 k2^k\rho\,c_q
 \prod_{r=1}^{k-1}\bigl[k(Nq+1)+Nr-k\bigr].
\]
Since
\[
 k(Nq+1)+Nr-k=N(kq+r),
\]
the Euler contribution is
\[
 k2^k\rho\,N^{k-1}c_q
 \prod_{r=1}^{k-1}(kq+r).
\]
Moreover,
\[
 \frac{(N(q+1)+1)!}{(Nq+1)!}
 =\prod_{\ell=2}^{2k}(Nq+\ell).
\]
Equation \eqref{eq:ratio} therefore says exactly that the derivative and
Euler contributions cancel.  Hence every coefficient vanishes.
\end{proof}

This completes the proof of Theorem \ref{thm:main}.

\section{Hypergeometric form and low-degree cases}

The coefficient ratio \eqref{eq:ratio} places the solution in the classical
generalized hypergeometric class.  With
\[
 z=-\frac{2^kk^k}{N^N}\rho x^N,
\]
one has
\begin{equation}\label{eq:hyp}
 \Gamma_{\rho,k}(x)
 =x\,{}_{k-1}F_{2k-2}\!\left(
 \begin{matrix}
  \frac1k,\frac2k,\ldots,\frac{k-1}{k}\\[1mm]
  \left\{\frac{j}{N}:2\leq j\leq N+1,\ j\ne N\right\}
 \end{matrix};z\right).
\end{equation}
Here the braces in the lower row denote the indicated list of $2k-2$
parameters.  Formula \eqref{eq:hyp} also makes differential finiteness
immediate in the standard sense of \cite{Stanley}; equation
\eqref{eq:mainode} gives the particularly transparent
operator adapted to the original variable.

For $k=2$, the quadratic boundary gives
\[
 \Gamma_{\rho,2}(x)
 =x\,{}_1F_2\!\left(\frac12;\frac23,\frac43;
 -\frac{16\rho x^3}{27}\right),
\]
and
\[
 \bigl[3\D^3+8\rho(2\vartheta+1)\bigr]\Gamma_{\rho,2}=0.
\]
For $k=3$,
\[
 \Gamma_{\rho,3}(x)
 =x\,{}_2F_4\!\left(\frac13,\frac23;
 \frac25,\frac35,\frac45,\frac65;
 -\frac{216\rho x^5}{3125}\right),
\]
and
\[
 \bigl[25\D^5+24\rho(3\vartheta+2)(3\vartheta+7)\bigr]
 \Gamma_{\rho,3}=0.
\]
For the quartic boundary,
\[
 \Gamma_{\rho,4}(x)
 =x\sum_{q\geq0}\frac{(-16\rho x^7)^q}{q!}
 \frac{(4q)!}{(7q+1)!},
\]
and the direct order-four guess is replaced by
\begin{equation}\label{eq:quartic7}
 \bigl[343\D^7+64\rho(4\vartheta+3)
 (4\vartheta+10)(4\vartheta+17)\bigr]\Gamma_{\rho,4}=0.
\end{equation}
Equivalently,
\[
 343\Gamma^{(7)}+4096\rho x^3\Gamma^{(3)}
 +43008\rho x^2\Gamma''+99072\rho x\Gamma'
 +32640\rho\Gamma=0.
\]
Thus the quartic obstruction is resolved by an uplift in differential order,
not by a correction inside the original sparse order-four ansatz.

For comparison, Table~\ref{tab:odes} records the first four members of the
hierarchy in fully expanded form.  In each row $G$ denotes
$\Gamma_{\rho,k}$.

\begin{table}[ht]
\centering
\caption{The first four constant-flux differential equations.}
\label{tab:odes}
\renewcommand{\arraystretch}{1.55}
\small
\begin{tabular}{c c >{\raggedright\arraybackslash}p{0.72\textwidth}}
\hline
$k$ & order & differential equation \\
\hline
$2$ & $3$ &
$3G^{(3)}+16\rho xG'+8\rho G=0$ \\
$3$ & $5$ &
$25G^{(5)}+216\rho x^2G''+864\rho xG'+336\rho G=0$ \\
$4$ & $7$ &
$343G^{(7)}+4096\rho x^3G^{(3)}+43008\rho x^2G''
+99072\rho xG'+32640\rho G=0$ \\
$5$ & $9$ &
$6561G^{(9)}+100000\rho x^4G^{(4)}
+2000000\rho x^3G^{(3)}+11440000\rho x^2G''
+19520000\rho xG'+5674240\rho G=0$ \\
\hline
\end{tabular}
\end{table}

The last equation is equivalently
\[
 \left[6561\D^9+160\rho(5\vartheta+4)(5\vartheta+13)
 (5\vartheta+22)(5\vartheta+31)\right]\Gamma_{\rho,5}=0.
\]

\section{An explicit compatibility-matrix calculation}\label{sec:matrixcheck}

The preceding proof obtains the differential hierarchy from the coefficient
ratio of the inverse construction.  We now reconnect it with the direct
compatibility-matrix mechanism of \cite{HDVMatrix}.  One complete example is
more informative than several large determinants, so we treat the quadratic
boundary in detail.

For $k=2$, the stationary operator in Table~\ref{tab:odes} is
\[
 L=3\D^3+16\rho x\D+8\rho.
\]
Using
\[
 P_txP_{-t}=x+t\D,
 \qquad P_t\D P_{-t}=\D,
\]
its transport is
\begin{equation}\label{eq:Nquadratic}
 N_t=P_tLP_{-t}
 =3\D^3+16\rho t\D^2+16\rho x\D+8\rho.
\end{equation}
Since $\D x=x\D+1$, successive left differentiation gives
\begin{align}
 \D N_t
 & =3\D^4+16\rho t\D^3+16\rho x\D^2+24\rho\D,
 \label{eq:DNquadratic}\\
 \D^2N_t
 & =3\D^5+16\rho t\D^4+16\rho x\D^3+40\rho\D^2.
 \label{eq:D2Nquadratic}
\end{align}

Now set $f(t)=\rho t^2$.  The first three homogeneous boundary operators,
with the same harmless row normalizations used in the compatibility-matrix
construction, are
\begin{align}
 P_1={}&\D^2+4\rho t\D,\label{eq:P1quadratic}\\
 P_2={}&\D^4+8\rho t\D^3+16\rho^2t^2\D^2+8\rho\D,
 \label{eq:P2quadratic}\\
 P_3={}&\frac12\D^6+6\rho t\D^5+24\rho^2t^2\D^4
 \notag\\
 &\quad +(12\rho+32\rho^3t^3)\D^3+48\rho^2t\D^2.
 \label{eq:P3quadratic}
\end{align}
They follow by differentiating $v(t,f(t))=0$ repeatedly along the curve and
using $v_t=\D^2v/2$.

Restrict \eqref{eq:Nquadratic}--\eqref{eq:D2Nquadratic} to
$x=\rho t^2$.  Since $v(t,f(t))=0$, the zeroth-order term can be discarded.
With columns ordered by the descending boundary jet
\[
 \mathbf V=
 \bigl(v_{xxxxxx},v_{xxxxx},v_{xxxx},v_{xxx},v_{xx},v_x\bigr)^{\mathsf T}
 \big|_{x=\rho t^2},
\]
and rows ordered as
\[
 \gamma_f(\D^2N_t),\ P_3,\ \gamma_f(\D N_t),\ P_2,
 \ \gamma_f(N_t),\ P_1,
\]
the compatibility system is $M_3[\rho t^2]\mathbf V=0$, where
\begin{equation}\label{eq:M3quadratic}
\boxed{
M_3[\rho t^2]=
\begin{pmatrix}
0&3&16\rho t&16\rho^2t^2&40\rho&0\\
\frac12&6\rho t&24\rho^2t^2&12\rho+32\rho^3t^3&48\rho^2t&0\\
0&0&3&16\rho t&16\rho^2t^2&24\rho\\
0&0&1&8\rho t&16\rho^2t^2&8\rho\\
0&0&0&3&16\rho t&16\rho^2t^2\\
0&0&0&0&1&4\rho t
\end{pmatrix}.}
\end{equation}

\begin{proposition}[Gaussian compatibility check]\label{prop:matrixcheck}
The matrix \eqref{eq:M3quadratic} is singular identically in $(t,\rho)$.
More precisely, Gaussian elimination produces a zero final row.
\end{proposition}

\begin{proof}
First permute the rows into the order
\[
 R_2,R_1,R_3,R_5,R_6,R_4.
\]
The first five rows are then already in echelon position.  Apply to the last
row the three fraction-free operations
\[
 R_6\longleftarrow3R_6-R_3,
 \qquad
 R_6\longleftarrow3R_6-8\rho tR_4,
 \qquad
 R_6\longleftarrow R_6+32\rho^2t^2R_5.
\]
The successive last rows are
\[
 (0,0,0,8\rho t,32\rho^2t^2,0),
\]
\[
 (0,0,0,0,-32\rho^2t^2,-128\rho^3t^3),
 \qquad (0,0,0,0,0,0).
\]
Thus the resulting upper-triangular matrix is
\[
\begin{pmatrix}
\frac12&6\rho t&24\rho^2t^2&12\rho+32\rho^3t^3&48\rho^2t&0\\
0&3&16\rho t&16\rho^2t^2&40\rho&0\\
0&0&3&16\rho t&16\rho^2t^2&24\rho\\
0&0&0&3&16\rho t&16\rho^2t^2\\
0&0&0&0&1&4\rho t\\
0&0&0&0&0&0
\end{pmatrix}.
\]
Consequently,
\[
 \det M_3[\rho t^2]=0.
\]
The calculation uses no division by $t$ or $\rho$, so the identity holds at
$t=0$ and $\rho=0$ as well.
\end{proof}

This is an independent direct verification of the operator discovered by
the inverse construction: only the proposed stationary operator and the
prescribed curve enter the matrix.  The cubic case produces a
$10\times10$ matrix and exact symbolic elimination again gives a zero
determinant.  The quartic and quintic cases have sizes $14$ and $18$;
displaying them would obscure rather than clarify the mechanism, and
Theorem~\ref{thm:main} already proves compatibility for every $k\geq2$.

\section{Nonuniqueness and the minimal-order question}

The hierarchy \eqref{eq:mainode} is attached to the constant-flux selection,
not to the boundary alone.  This distinction is essential.  The inverse
kernel $\mathcal K(\rho t^k)$ contains many formal elements, and distinct
normalizations can have different differential properties
\cite{HDVInverse}.  In low degrees, the same quadratic and cubic boundaries
also arise from stationary equations of orders two and three, respectively,
whereas the present construction produces orders three and five.
Consequently, the phrase ``the differential order of a boundary'' is not
intrinsic unless a selection principle is specified.

There are at least two different minimality problems.

\begin{question}
For fixed $k$ and $\rho\ne0$, what is the minimal order of a
polynomial-coefficient differential operator annihilating the particular
constant-flux function $\Gamma_{\rho,k}$?
\end{question}

\begin{question}
What is the minimal differential order attained by any nonzero analytically
admissible element of $\mathcal K(\rho t^k)$?
\end{question}

Theorem \ref{thm:main} supplies upper bounds for both questions, but it does
not answer either one.  More broadly, it suggests a change in viewpoint:
instead of guessing a stationary ODE of the same order as the degree of the
desired boundary, first prescribe a boundary trace and flux, solve the
inverse heat problem, and then extract the differential equation satisfied
by the selected initial function.

\section{Conclusion}

For every monomial curve $x=\rho t^k$, the inverse heat problem with zero
trace and unit spatial flux has an explicit entire solution.  Its initial
function is hypergeometric and satisfies a polynomial-coefficient ODE of
order $2k-1$.  The quartic case is the first instance in which this mechanism
decisively departs from the naive order-$k$ guess: the correct constant-flux
selection is governed by an order-seven equation.

The conclusion is not that $2k-1$ is always the least possible order.
Rather, the prescribed boundary determines an inverse kernel, and a further
selection principle determines which differential structure becomes
visible.  This separation between boundary geometry and stationary
realization appears to be the organizing principle behind the hierarchy.

\end{document}